\documentclass{amsart}
\usepackage{amsmath}
\usepackage[dvips]{graphicx}
\usepackage[usenames,dvipsnames]{xcolor}
\usepackage{colortbl}
\usepackage{multirow}
\usepackage{pgf,tikz}
\usepackage{subfig}
\usepackage{amsmath, amsthm, amscd, amsfonts, amssymb, color}
\usepackage[bookmarksnumbered, plainpages]{hyperref}
\usepackage{relsize}
\usepackage{longtable}
\newtheorem{thm}{Theorem}[section]

\newtheorem{con}[thm]{Conjecture}
\newtheorem{lem}[thm]{Lemma}
\newtheorem{prop}[thm]{Proposition}
\newtheorem{defn}[thm]{Definition}
\newtheorem{rem}[thm]{Remark}

\newtheorem{qu}[thm]{Question}
\numberwithin{equation}{section}
\DeclareMathOperator{\cay}{Cay}
\DeclareMathOperator{\mat}{Mat}
\DeclareMathOperator{\GL}{GL}

\DeclareMathOperator{\Stab}{Stab}
\makeatletter
\@namedef{subjclassname@2020}{%
  \textup{2020} Mathematics Subject Classification}
\makeatother
\begin{document}
\title[Permutation Codes under  Kendall $\tau$-Metric]{New Upper Bounds on the Size of  Permutation Codes under Kendall $\tau$-Metric}
\author[Abdollahi]{Alireza Abdollahi}%
\address{Department of Pure Mathematics, Faculty of Mathematics and Statistics,  University of Isfahan, Isfahan 81746-73441, Iran.}%
\address{School of Mathematics, Institute for Research in Fundamental Sciences (IPM), 19395-5746 Tehran, Iran.}
\email{a.abdollahi@math.ui.ac.ir}%
\author[Bagherian]{Javad Bagherian}
\address{Department of Pure Mathematics, Faculty of Mathematics and Statistics,  University of Isfahan, Isfahan 81746-73441, Iran.}%
\email{bagherian@sci.ui.ac.ir}
\author[Jafari]{Fatemeh Jafari}
\address{Department of Pure Mathematics, Faculty of Mathematics and Statistics,  University of Isfahan, Isfahan 81746-73441, Iran.}%
\email{math\_fateme@yahoo.com}
\author[Khatami]{Maryam Khatami}%
\address{Department of Pure Mathematics, Faculty of Mathematics and Statistics,  University of Isfahan, Isfahan 81746-73441, Iran.}%
\email{m.khatami@sci.ui.ac.ir}
\author[Parvaresh]{Farzad Parvaresh}
\address{Department of Electrical Engeenering,  University of Isfahan, Isfahan 81746-73441, Iran.}%
\address{School of Mathematics, Institute for Research in Fundamental Sciences (IPM), 19395-5746 Tehran, Iran.}
\email{f.parvaresh@eng.ui.ac.ir}
\author[Sobhani]{Reza Sobhani}%
\address{Department of Applied Mathematics and Computer Science, Faculty of Mathematics and Statistics,  University of Isfahan, Isfahan 81746-73441, Iran.}%
\email{r.sobhani@sci.ui.ac.ir}
\thanks{Corresponding Author: A. Abdollahi (a.abdollahi@math.ui.ac.ir)}
\thanks{F. Parvaresh is supported by IPM in part by grant No. 1401680050.}

\begin{abstract}
We first give two methods based on the representation theory of symmetric groups to study the largest size $P(n,d)$ of  permutation codes of length $n$ i.e. subsets of the set $S_n$ all permutations on $\{1,\dots,n\}$ with the minimum distance (at least) $d$ under the Kendall $\tau$-metric. The first method is an integer programming problem obtained from the transitive actions of $S_n$. The second method can be applied to refute the existence of perfect codes in $S_n$.\\    
Here we reduce the known upper bound $(n-1)!-1$ for $P(n,3)$ to $(n-1)!-\lceil\frac{n}{3}\rceil+2\leq (n-1)!-2$, whenever $n\geq 11$ is any prime number. If $n=6$, $7$, $11$, $13$, $14$, $15$, $17$,   the known upper bound for $P(n,3)$ is decreased by $3,3,9,11,1,1,4$,  respectively.
\end{abstract}
\maketitle
\section{Introduction}
Rank modulation was proposed as a solution to the
challenges posed by flash memory storages \cite{jiang}. In the rank
modulation framework, codes are permutation codes, where by a Permutation Code (PC) of length $n$ we simply mean a non-empty subset of $S_n$, the set of all permutations of $[n]:=\{1,2,\ldots ,n\}$.
 Given a permutation $\pi:=[\pi(1),\pi(2),\ldots ,\pi(i),\pi(i+1),\ldots ,\pi(n)]\in S_n$,  an adjacent transposition, $(i, i + 1)$, for some $1\leq i\leq n-1$,  applied to $\pi$ will result in the permutation $[\pi(1),\pi(2),\ldots ,\pi(i+1),\pi(i),\ldots ,\pi(n)]$. For two permutations
$\rho,\pi\in S_n$, the Kendall $\tau$-distance between $\rho$ and $\pi$,
$d_K(\rho, \pi)$, is defined as the minimum number of adjacent
transpositions needed to write   $\rho \pi^{-1}$ as their product. Under the Kendall $\tau$-metric a PC of length $n$ with minimum distance $d$ can correct up to $ \frac{d-1}{2} $ errors caused by charge-constrained errors  \cite{jiang}.

 The maximum size of a PC of length $n$ and   minimum Kendall $\tau$-distance  \textit{at least}  $d$ is denoted by $P(n,d)$ and a PC attaining this size is said to be optimal. We will show in Proposition \ref{atleast}, below, that  if $d$ is such that optimal PCs of minimum Kendall $\tau$-distance at least $d$ exist, then the ones with the minimum distance  exactly $d$ exist and so one can drop the condition ``at least" in the latter definition for the minimum distance. Several researchers
have presented bounds on $P(n, d)$ (see \cite{barg,BE,jiang,V,WZYG,WWYF}), some of these results are shown in Table \ref{10101}. It is known that $P(n, 1)=n!$ and $P(n,2)=\frac{n!}{2}$. Also it is known that if $\frac{2}{3}\binom{n}{2}<d\leq \binom{n}{2} $, then $P(n,d)=2$ (see \cite[Theorem 10]{BE}). However, determining $P(n, d)$ turns out to be difficult for $3\leq d\leq \frac{2}{3}\binom{n}{2} $. In this paper, we study the upper bound of $P(n,3)$. By sphere packing bound (see \cite[Theorems 12 and 13]{jiang})  $P(n,3)\leq (n-1)!$. A PC  of size $(n-1)!$ and  with minimum Kendall $\tau$-distance 3 is called a 1-perfect code. It is proved that if $n>4$ is a prime number or $4\leq n\leq 10$, then there is no 1-perfect code in $S_n$ (see \cite[Corollary 2.5  and Theorem 2.6]{white} or \cite[Corollary 2]{BE}).

\begin{table}[!tbp]
	\begin{tabular}{|c|c|c|c|c|}
		\hline
		\cellcolor{yellow!50}{$n$}&\cellcolor{yellow!50}{6}&\cellcolor{yellow!50}{7} &\cellcolor{yellow!50}{11}&\cellcolor{yellow!50}{13}\\
		\hline
		Old UB& $5!-1^{a}$ & $6!-1^{a}$ & $10!-1^{a}$ & $12!-1^{a}$ \\
		\hline
		UB & \cellcolor{black!20!white}{$5!-4$}&\cellcolor{black!20!white}{$6!-4$}&\cellcolor{black!20!white}{$10!-10$}&\cellcolor{black!20!white}{$12!-12$}\\
		\hline
		\cellcolor{yellow!50}{$n$}&\cellcolor{yellow!50}{14}&\cellcolor{yellow!50}{15} &\cellcolor{yellow!50}{17}&\cellcolor{yellow!50}{prime  $n\geq 19$ }\\
		\hline
		Old UB& $13!$ \cite{jiang}&$14!$ \cite{jiang}&$ 16!-1^{a} $ &$(n-1)!-1^{a}$ \\
		\hline
		UB&\cellcolor{black!20!white}{$13!-1$}&\cellcolor{black!20!white}{$14!-1$}&\cellcolor{black!20!white}{$16!-5$}&\cellcolor{black!20!white}{$(n-1)!-\lceil\frac{n}{3}\rceil +2$}\\
		\hline
	\end{tabular}
	\caption{  {\small Some results on the upper bounds of $P(n,3)$.  The superscripts show the references from which the upper bound is taken, where ``a" is \cite{BE,white}, and  gray color shows our main results.}}\label{10101}
\end{table}

There are several works using optimization techniques to bound
the size of permutation codes under various distance metrics
(Hamming, Kendall $\tau$, Ulam) (see \cite{Go,tar,V}).  In  Section \ref{sec2}, we show that for any non-trivial subgroup of $S_n$, we can derive an integer programming problem where the optimal
value of the objective function  gives an upper bound on $P(n,3)$. In Section \ref{sec3}, by considering the integer programming problem corresponding to the Young  subgroups (see Definition \ref{young1}, below) of  $S_n$,
 we  prove the following result:
 \begin{thm}\label{mainresult}
For all primes $p\geq 11$,   $P(p,3)\leq (p-1)!-\lceil \frac{p}{3}\rceil+2\leq (p-1)!-2$.
 \end{thm}
\noindent We then use a software to solve the integer programming problems that are derived from specific choices of the underline subgroup and obtain tighter upper bounds for some small values of $n$. Finally, we apply a related method from
\cite{white} to prove the nonexistent of 1-perfect codes in $S_{14}$, $S_{15}$. 

\section{Preliminaries}\label{sec2}
A \emph{simple graph} $\Gamma$ consists of a non-empty set of vertices $V (\Gamma)$ and a possibly empty  set of edges $E(\Gamma)$ which is a subset of the set of all 2-element subsets of $V (\Gamma)$. Two vertices $\sigma_1$ and $\sigma_2$ are called adjacent, denoted by $ \sigma_1 -\sigma_2 $, if $\{\sigma_1,\sigma_2\} \in E(\Gamma)$. A \emph{subgraph}  $H$ of $\Gamma$ is a simple graph  whose vertex set and edge set are subsets of those of $\Gamma$. A \emph{path}  is a simple graph with the vertex set $\{\sigma_0, \sigma_1, \ldots , \sigma_n\}$ such that $\sigma_{j} - \sigma_{j+1}$ for $j = 0,\ldots, n-1$. The length of a path is the number of its edges. 

By a \emph{graphical code} of minimum distance at least $d$ we mean a subset of vertices of a simple graph such that any two distinct vertices has distance at least $d$, where the distance of two vertices is defined to be the shortest length of a  path between the vertices.
 Examples of such codes are permutation codes under Kendall $\tau$-metric  or Ulam metric. In fact the set of all permutations with the Kendall $\tau$ or Ulam
metrics can be represented as Cayley graph (see Definition \ref{cayley}, below) and  PCs are then subgraphs of the Cayley graph. The methods used in this paper rely on the fact that the permutation set with Kendall $\tau$-metric is a Cayley graph.

Here we observe that if $d$ is such that  graphical codes of minimum distance at least $d$ exist, then the ones with the minimum distance exactly $d$ exist. 

\begin{prop}\label{atleast}
Let $\Gamma$ be any simple graph and $d\geq 1$ an integer. Then $$\{|C| \;|\; C\subseteq V(\Gamma) \; \text{\rm and} \; d_\Gamma(C)=d\}=\{|C| \;|\; C\subseteq V(\Gamma) \; \text{\rm and} \; d_\Gamma(C)\geq d\},$$
where $d_\Gamma(C)=\min\{d_\Gamma(x,y)\;|\; x,y\in C \; \text{\rm and} \; x\neq y\}$. 
\end{prop}

\begin{proof}
Let $C$ be a graphical code with the minimum distance at least $d$. Suppose that $\sigma,\tau \in C$ such that $d_\Gamma(C)=d_\Gamma(\sigma,\tau)=d+\ell$ for some non-negative integer $\ell$. If $\ell=0$, we are done; so from  now on assume that $\ell>0$. Let $\sigma-\sigma_1-\cdots-\sigma_\ell-\cdots-\sigma_{d+\ell-1}-\tau$ be a shortest path in the graph $\Gamma$ between $\sigma$ and $\tau$. Consider 
$\hat{C}=(C\setminus \{\sigma\}) \cup \{\sigma_\ell\}$. We claim that $|C|=|\hat{C}|$ and $d_\Gamma(\hat{C})=d$, this will complete the proof. 
If $\sigma_\ell \in C$, then $d(\sigma_\ell, \tau ) = d$, which implies  $\ell=0$, a contradiction. It follows that $|C|=|\hat{C}|$. To prove that $d_\Gamma(\hat{C})=d$, it is enough to show that $d_\Gamma(\delta,\sigma_\ell)\geq d$ for all $\delta\in C\setminus\{\sigma\}$. Since $d_\Gamma(C)=d+\ell$ and by the triangle inequality we have
$$d+\ell\leq d_\Gamma(\delta,\sigma)\leq d_\Gamma(\delta,\sigma_\ell)+d_\Gamma(\sigma_\ell,\sigma)=d_\Gamma(\delta,\sigma_\ell)+\ell.$$
So $d_\Gamma(\delta,\sigma_\ell)\geq d$, as required.
\end{proof} 
A PC with Hamming metric is not a graphical code as the Hamming distance between two permutations is never equal to 1 and so we cannot apply Proposition \ref{atleast} for the latter case. We do not know if the conclusion of Proposition of \ref{atleast} is valid for PCs with Hamming metric. We propose the following question. 
\begin{qu}
Let $d_H$ be the Hamming metric on $S_n$ and $d\geq 2$ be an arbitrary integer. Is it true that $$\{|C| \;|\; C\subseteq S_n \; \text{\rm and} \; d_H(C)=d\}=\{|C| \;|\; C\subseteq S_n \; \text{\rm and} \; d_H(C)\geq d\}?,$$
where $d_H(C)=\min\{d_H(x,y)\;|\; x,y\in C \; \text{\rm and} \; x\neq y\}$.
\end{qu}
\begin{defn}
Let $G$ be a finite group and $B, C$ be two non-empty subsets of $G$.  As usual we denote by $BC$ the set $\{bc\,|\,b\in B, c\in C\}$, where by $g=bc$ we refer to the group operation, also for each $g\in G$ we denote by $Bg$ the set $B\{g\}$. The set $B$ is called inverse closed if $B=B^{-1}:=\{b^{-1}\,|\,b\in B\}$. Also, we use the notation  $\xi$ to denote  the identity element of $G$. 
\end{defn} 
Let $G$ be a finite group and denote by $\mathbb{C}[G]$ the ``complex group algebra" of $G$. The elements of $\mathbb{C}[G]$ are of the formal sum 
\begin{equation}
\label{elementgroupring} 
\sum_{g \in G} a_g g,
\end{equation}
where $a_g\in \mathbb{C}$.  The complex group algebra is a $\mathbb{C}$-algebra with the following addition, multiplication and scaler product:
$$\sum_{g \in G} a_g g + \sum_{g \in G} b_g g=\sum_{g \in G} (a_g+b_g) \sigma, $$
$$\big(\sum_{g \in G} a_g g \big)  \big(\sum_{g \in G} b_g g\big)=\sum_{g \in G} \big(\sum_{g=g_1 g_2}a_{g_1} b_{g_2}\big) g,$$
$$\lambda \sum_{g \in G} a_g g =\sum_{g \in G}(\lambda a_g) g,$$
where $\lambda,a_g,b_g \in \mathbb{C}$.
If $a_g=0$ for some $g$, the term $a_g g$ will be neglected in \ref{elementgroupring} and $\sum_{g \in G} a_g g$ is written as
$a_1 g_1 +\cdots+a_k g_k$, where $\{g \;|\; a_g\neq 0\}=\{g_1,\dots,g_k\}$ is non-empty and otherwise $\sum_{g \in G} a_g g$ is denoted by $0$. For a non-empty finite subset $\Theta$ of $G$, we denote by $\widehat{\Theta}$ the element $\sum_{\theta\in \Theta} \theta$ of $\mathbb{C}[G]$.

 \begin{defn}\label{cayley}
 Let $G$ be a finite group and $S$ be a non-empty inverse closed subset of $G$  not-containing the identity element $\xi$ of $G$. Then the Cayley graph $\Gamma:=\cay(G,S)$ is a simple graph with $V(\Gamma)=\{g\,|\,g\in G\}$ and $E(\Gamma)=\big\{\{g,h\}\,\big|\,g,h\in G,gh^{-1} \in S\big\}$. 
 \end{defn}
  Let $G$ be a finite group and $S$ be a non-empty inverse closed subset of $G$  not-containing the identity element $\xi$ of $G$. Now we have a metric $d_\Gamma$ on $G$ defined by $\Gamma$ which is the shortest length of a  path between two vertices in $\cay(G,S)$. For example if $G=S_n$ and $S=\{(1,2),(2,3),\dots,(n-1,n)\}$, the metric $d_\Gamma$ is the Kendall $\tau$-metric on $S_n$. Also if $G=S_n$ and $S=T \cup T^{-1}$, where $T:=\{ (a , a+1, \ldots , b) \;|\; a<b, a,b\in [n] \}$, the metric $d_\Gamma$ is the Ulam metric on $S_n$. 
\begin{defn}
For a positive integer $r$ and an element $g\in G$, the ball of radius $r$ in $G$ under the metric $d_\Gamma$ is denoted by $B_r^\Gamma(g)$ defined by 
$B_r^\Gamma(g)=\{h\in G \;|\; d_\Gamma(g,h)\leq r\}$.
\end{defn}
\begin{rem}\label{ballsize}
 Note that $B_r^{\Gamma}{(g)}=(S^r \cup\{\xi\})g$, where $S^r:=\{s_1\cdots s_t \;|\; s_1,\dots,s_t\in S,\,1\leq t\leq r\}$. Also note that since $S$ is inverse closed, $B_r^{\Gamma}{(g)}=S^r g$ for all $r\geq 2$. It follows that $|B_r^\Gamma(g)|=|B_r^\Gamma(1)|=|S^r \cup\{\xi\}|$ for all $g\in G$. 
\end{rem}
\begin{prop}\label{rel}
Let $G$ be a finite group and $d_{\Gamma}$ be the metric induced by the graph $\cay(G,S)$. Then a subset $C$  of  $G$ is a code with $min\{d_{\Gamma}(x,y)\,|\,x,y\in C\}\geq d$ if and only if there exists $Y\subset G$ such that
	\begin{equation} \label{groupring}
	\widehat{(S^{\lfloor\frac{d-1}{2}\rfloor} \cup \{\xi\})} \widehat{C} =\widehat{G}-\widehat{Y},
	\end{equation}  
\end{prop}
\begin{proof}
Let $r:=\lfloor\frac{d-1}{2}\rfloor$, $Y=G\setminus\cup_{c\in C}B_r^\Gamma(c)$ and $T:=S^r\cup\{\xi\}$. So $G=\cup_{c\in C}B_r^\Gamma(c)\cup Y$. It follows from   Remark \ref{ballsize} that for each $c\in C$, $B_r^\Gamma(c)=Tc$ and so $\cup_{c\in C}B_r^\Gamma(c)=TC$. Therefore, $\widehat{G}=\widehat{T C}+\widehat{Y}$. On the other hand, for any two distinct elements $c,c'$ in $C$, $Tc \cap T c'=\varnothing$ since otherwise $d_{\Gamma}(c,c')\leq d-1$ that is a contradiction. Hence, $ \widehat{T C}=\widehat{T}\widehat{C}$ and this completes the proof.
\end{proof}
\begin{defn}
Let $G$ be a finite group and $d_{\Gamma}$ be the metric induced by  $\cay(G,S)$. For a positive integer $r$, an $r$-perfect code or a perfect code of radius $r$ of $G$ under the metric $d_\Gamma$ is a subset $C$ of $G$ such that $G=\cup_{c\in C} B_r^\Gamma(c)$ and $B_r^\Gamma(c)\cap B_r^\Gamma(c')=\varnothing$ for any two distinct $c,c'\in C$.
\end{defn}
\begin{rem}
By a similar argument as the proof of Proposition {\rm\ref{rel}}, it can be seen that if $C$ is an $r$-perfect code, then $\widehat{(S^r \cup \{\xi\})} \widehat{C} =\widehat{G}$. We note that according to Remark {\rm\ref{ballsize}} $C$ is an $r$-perfect code if and only if  $|C||S^r\cup \{\xi\}|=|G|$. 
\end{rem}
Let $\rho$ be any (complex) \emph{representation} of a finite group $G$ of dimension $k$ for some positive integer $k$, i.e., any group homomorphism from $G$ to the general linear group $\GL_k(\mathbb{C})$ of  $k\times k$ invertible matrices over $\mathbb{C}$. Then by the universal property of $\mathbb{C}[G]$, $\rho$ can be extended to an algebra homomorphism $\hat{\rho}$ from $\mathbb{C}[G]$ to the algebra $\mat_k(\mathbb{C})$ of $k\times k$ matrices over $\mathbb{C}$ such that
$g^{\hat{\rho}}=g^\rho$ for all $g\in G$. Thus the image of $\widehat{\Theta}$ for any non-empty subset $\Theta$ of $G$ under $\hat{\rho}$ is  the element $\sum_{\theta\in \Theta}\theta^\rho$ of $\mat_k(\mathbb{C})$. 
In particular  by applying $\hat{\rho}$ on the equality \ref{groupring}, we obtain 	  
\begin{equation}\label{groupringmain}
	\big(\sum_{s\in S \cup\{\xi\}} s^\rho\big) \big(\sum_{c\in C}c^\rho\big)=\sum_{g\in G}g^\rho- \sum_{y\in Y }y^\rho, 
\end{equation} where
	the latter equality is between elements of $\mat_k(\mathbb{C})$. 
	
In the following, we state an important definition that it will play a central role in the proof of the main results of this paper.
\begin{defn}\label{action}
Given a group $G$ and a non-empty set $\Theta$, recall that we say $G$ acts on $\Theta$ {\rm(}from the right{\rm)} if there exists a function $\Theta\times G\rightarrow \Theta$ denoted by $(\theta,g)\mapsto \theta^g$ for all $(\theta,g)\in \Theta\times G$ if $(\theta^g)^h=\theta^{gh}$ and $\theta^\xi=\theta$ for all $\theta\in \Theta$ and all $g,h\in G$.  For any $\theta\in \Theta$ the set 
$\Stab_G(\theta):=\{g\in G\;|\; \theta^g=\theta\}$ is called the stabilizer of $\theta$ in $G$ which is a subgroup of $G$. If the action is transitive {\rm(}i.e., for any two elements $\theta_1,\theta_2 \in \Theta$, there exists $g\in G$ such that $\theta_1^g=\theta_2${\rm)}, all stabilizers are conjugate under the elements of $G$, more precisely $\Stab_G(\theta_1)^g=\Stab_G(\theta_2)$ whenever $\theta_1^g=\theta_2$, where $\Stab_G(\theta_1)^g=g^{-1}\Stab_G(\theta_1)g$. \\
Now suppose that $G$ acts on $\Theta$ and $|\Theta|=k$ is finite.  Fix an arbitrary ordering on the elements of $\Theta$ so that $\theta_i<\theta_j$ whenever $i<j$ for distinct elements $\theta_i,\theta_j\in \Theta$. Denote by $\rho_{\Theta}^G$ the map from $G$ to $\GL_k(\mathbb{Z})$ defined by $g\mapsto P_g$, where
$P_g$ is the $|\Theta|\times |\Theta|$ matrix whose $(i,j)$ entry is $1$ if $\theta_i^g=\theta_j$ and $0$ otherwise. 
\end{defn}
\begin{rem}
Note that the definitions of $\rho_{\Theta}^G$ depends on the choice of the ordering on $\Theta$, however any two such representations of $G$ are conjugate by a permutation matrix.
\end{rem}
\begin{rem}
Let $H$ be a subgroup of a finite group $G$ and  $X$ be the set of  right cosets of $H$ in $G$, i.e., $X:=\{Hg\,|\, g\in G\}$. Then $G$ acts transitively on $X$ via $(Hg,g_0)\longrightarrow Hgg_0$. It is known that $X$  partitions  $G$, i.e., $G=\cup_{x\in X}{x}$ and $x\cap x'=\varnothing$ for all distinct elements $x$ and $x'$ of $X$, and $|X|=|G|/|H|$.
\end{rem}
\begin{lem}\label{coset1}
Let $H$ be a subgroup of a finite group $G$ and  $X=\{Ha_1,\ldots ,Ha_m\}$ be the set of  right cosets of $H$ in $G$. If $\mathcal{Y}\subset G$, then by fixing the ordering  $Ha_i<Ha_j$ whenever $i<j$, the $(i,j)$ entry of  $\sum_{y\in\mathcal{Y}}y^{\rho_{X}^G}$  is $|\mathcal{Y}\cap {a_i}^{-1}Ha_j| $. 
\end{lem}
\begin{proof}
Clearly, for any $y\in \mathcal{Y}$, the $(i,j)$ entry of  $y^{\rho_{X}^G}$ is $1$ if $Ha_iy=Ha_j$ and is $0$ otherwise. So the $(i,j)$ entry of  $y^{\rho_{X}^G}$ is $1$ if $a_i y {a_j}^{-1}\in H$ and therefore $y\in {a_i}^{-1}Ha_j$. Hence, the $(i,j)$ entry of  $\sum_{y\in\mathcal{Y}}y^{\rho_{X}^G}$ is equal to $|\{y\in \mathcal{Y}\, |\,y\in {a_i}^{-1}Ha_j\}|$. This completes the proof. 
\end{proof}
The following result summarizes the main method used in this paper.
\begin{thm}\label{intprogram}
Let $G$ be a finite group and $d_{\Gamma}$ be the metric induced by the graph $\cay(G,S)$. Also, let $C$ be a code in $G$ with $min\{d_{\Gamma}(c,c')\,|\,c,c'\in C\}\geq d$. If $H$ is a subgroup of $G$ and $X$ is the set of  right cosets of $H$ in $G$, then the optimal value of the objective function of the following integer programming problem gives an upper bound on $|C|$.
\begin{align*}\label{max}
\text{Maximize} &\quad \sum_{i=1}^{|X|}{x_i},\\
\text{subject to}&\quad \widehat{T^{\rho_{X}^G}}(x_1,\ldots ,x_{|X|})^t \leq |H| \mathbf{1},\\
& \quad x_i\in \mathbb{Z},\,\, x_i\geq 0, \,\, i\in\{1,\ldots,|X|\},
\end{align*}
where $T:=S^{\lfloor\frac{d-1}{2}\rfloor}\cup \{\xi\}$, $\textbf{1}$ is the column vector of order $|X|\times 1$ whose  entries are equal to $1$. 
\end{thm}
\begin{proof}
Let $r:=\lfloor\frac{d-1}{2}\rfloor$. By Equation \ref{groupringmain}, there exists $Y\subset G$ such that
\begin{equation}
	\big(\sum_{s\in T} s^{\rho_{X}^G}\big) \big(\sum_{c\in C}c^{\rho_{X}^G}   \big)=\sum_{g\in G}g^{\rho_{X}^G}- \sum_{y\in Y }y^{\rho_{X}^G}, 
	\end{equation}
	Suppose that $X=\{Ha_1,\ldots ,Ha_m\}$. Without loss of generality, we may assume that $a_1=1$. We fix the ordering $Ha_i<Ha_j$ whenever $i<j$.
	By Lemma \ref{coset1}, the $(i,j)$ entry of $\sum_{g\in G}g^{\rho_{X}^G}$ is equal to $|G \cap {a_i}^{-1}Ha_j|$ and  since  $ {a_i}^{-1}Ha_j\subseteq G$, the $(i,j)$ entry of $\sum_{g\in G}g^{\rho_{X}^G}$ is equal to $|{a_i}^{-1}Ha_j|=|H|$. So if $B$ is a column of $\sum_{g\in G}g^{\rho_{X}^G} $, then $B=|H| \mathbf{1}$. Let $\mathcal{C}$ be the first column of $\sum_{c\in C}c^{\rho_{X}^G}$. Then  Lemma \ref{coset1} implies that for all $1\leq i\leq |X|$,  $i$-th row of $\mathcal{C}$, denoted by $c_i$, is equal to $|C\cap Ha_i|$. Since $C=C\cap G=\cup_{i=1}^{|X|} (C\cap Ha_i)$ and $(C\cap Ha_i)\cap (C\cap Ha_j)=\varnothing$ for all $i\neq j$, $\sum_{i=1}^{|X|}c_i=|C|$. We note that by Lemma \ref{coset1}, all entries of matrix $\widehat{F^{\rho_{X}^G}}$, $F\in \{C,G,Y,T\}$, are integer and non-negative. Therefore $\mathcal{C}$ is an integer solution for the following system of inequalities 
	\[\widehat{T^{\rho_{X}^G}}(x_1,\ldots ,x_{|X|})^t \leq |H| \mathbf{1}
	\]
	such that $\sum_{i=1}^{|X|}c_i=|C|$ and this completes the proof.
\end{proof}

\section{Results}\label{sec3}
Let $G=S_n$ and $S=\{(i,i+1)\,|\,1\leq i\leq n-1\}$, then the metric induced by $\cay(G,S)$ on $S_n$ is the Kendall $\tau$-metric. In this section, by using the result in Section \ref{sec2}, we improve the upper bound of $P(n,3)$ when  $n\in \{6,14,15\} $ or $n\geq 7$ is a prime number. We note that for two permutations $\sigma$ and $\lambda$ of $S_n$, their multiplication $\lambda \cdot\sigma$ is defined as the composition of $\sigma$ on $\lambda$, namely  $\lambda\cdot \sigma(i) = \sigma(\lambda(i))$ for all $i \in [n]$.

In order to apply Theorem \ref{intprogram}, we need to fix the subgroup
$H$ and that different choices for $H$ will lead to different results.
Usual traditional with well-developed candidate for $H$ is the  Young subgroups  which we are going to recall them \cite{JK}. 
\begin{defn}\label{young1}
     By a number  partition $\lambda$ of $n$ {\rm(}with the length $m)$ we mean an $m$-tuple $(\lambda_1,\dots,\lambda_m)$ of positive integers such that $\lambda_1\geq \cdots\geq\lambda_m$ and $n=\sum_{i=1}^m \lambda_i$. If $\lambda$ and $\mu$ are two partitions of $n$, we say that $\lambda$ dominates $\mu$, and write $\lambda\unlhd \mu$, provided that $\sum_{i=1}^j \lambda_i\geq \sum_{i=1}^j \mu_i$ for all $j$. Let $\lambda$ be a partition of $n$ and $\Delta:=(\Delta_1,\dots,\Delta_m)$ be an $m$-tuple  of non-empty subsets of $[n]$  consisting a set partition  for $[n]$ with $|\Delta_i|=\lambda_i$ for all $i=1,\dots,m$. 
We associate a Young subgroup $S_\Delta$ of $S_n$  by taking
 $S_\Delta=S_{\Delta_1}\times \cdots\times S_{\Delta_m}$, where $S_{\Delta_i}$ is the symmetric group on the set $\Delta_i$ for all $ i=1,\ldots ,m$. 
\end{defn}
\begin{rem}\label{young}
Let $\lambda$ be a partition of $n$ and $\Delta$, $\Delta'$  be two $m$-tuples  of non-empty subsets of $[n]$  consisting a set partition  for $[n]$ with $|\Delta_i|=|\Delta'_i|=\lambda_i$ for all $i=1,\dots,m$. It is known that the representations $\rho_{X}^{S_n}$ and $\rho_{X'}^{S_n}$, where $X$ and $X'$ are the set of  right cosets of the Young subgroups $S_{\Delta}$ and $S_{\Delta'}$ in $S_n$, respectively,  are equivalent {\rm(}i.e.,  a matrix $U$ exists such that $U^{-1}\rho_{X}^{S_n}(\sigma)U=\rho_{X'}^{S_n}(\sigma)$ for all $\sigma \in S_n)$. Hence, we use the $m$-tuples  of non-empty subsets of $[n]$, $[\{1,\ldots ,\lambda_1\},\{\lambda_1+1,\ldots ,\lambda_1+\lambda_2\},\ldots,\{n-\lambda_m+1,\ldots ,n\}]$ for considering the Young subgroup corresponding to the partition $\lambda=(\lambda_1,\ldots ,\lambda_m)$,  as we are studying these representations up to equivalence. 
\end{rem}
For example, if $n = 7$ and $\lambda = (3, 2, 2)$, then  the  Young subgroup corresponding to the partition $\lambda$ is the subgroup $H = \{\sigma_1\cdot \sigma_2\cdot ·\sigma_3 \,|\, \sigma_1 \in S_3, \sigma_2 \in S_{\{4,5\}}, \sigma_3 \in S_{\{6,7\}}\}$. 
\begin{lem}\label{matrix}
Let $H$ be a Young subgroup of $S_n$ corresponding to the partition $\lambda:=(n-1,1)$ and $X$ be the set of  right cosets of $H$ in $S_n$. If $S=\{(i,i+1)\,|\,1\leq i\leq n-1\}$ and $T:=S\cup \{\xi\}$,  then $ \widehat{T^{\rho_{X}^{S_n}}} $ is a  conjugate by a permutation matrix of the following matrix
\begin{equation}\label{matn}
\begin{pmatrix}
	n-1   &1    &0     &0     &\dots &0  \\
	1     &n-2  &1     &0     &\dots &0    \\
	0     &1    &n-2   &1     &0     &0    \\
	\vdots&\dots&\ddots&\ddots&\ddots&\vdots\\
	0     &0    &\dots &1     &n-2   &1     \\
	0     &0    &\dots &0     &1     &n-1
	\end{pmatrix}.
\end{equation}
\end{lem}
\begin{proof}
 Without loss of generality we may assume that $\lambda$ is the partition $\{\{1\},\{2,\ldots ,n\}\}$ of $n$ and therefore $H=\Stab_{S_n}(1)$. Clearly, for each $i\in [n]$, if $\sigma\in H(1,i)$, then $\sigma(1)=i$ and so $H(1,i)\cap H(1,j)=\varnothing$  for all $i\neq j$.  So we can let  $X=\{H(1,i)\,|\,1\leq i\leq n\}$, where we are using the convention $H(1,1):=H$. Fix the ordering of $X$ such that $H(1,i)<H(1,j)$ if $i<j$. By Lemma \ref{coset1}, the $(i,j)$ entry of $ \widehat{T^{\rho_{X}^{S_n}}} $ is equal to $|T\cap (1,i)H(1,j)|$. If $i=j$, then Remark \ref{action} implies $ (1,i)H(1,i)=\Stab_{S_n}(i) $ and hence $T\cap (1,i)H(1,i)=T\setminus \{(i-1,i),(i,i+1)\} $ if $2\leq i\leq n-1$, $T\cap (1,n)H(1,n)=T\setminus \{(n,n-1)\} $ and $T\cap H=T\setminus \{(1,2)\} $. Now suppose that $i\neq j$. Clearly $(1,i)\cdot(i,j)\cdot(1,j)=(i,j)$. Let $h\in H$. Then  $\sigma:=(1,i)\cdot h\cdot(1,j)=\pi(1,j,i)$, where $\pi=(1,i)\cdot h\cdot(1,i)\in \Stab_{S_n}(i)$. Since $\pi(i)=i$, $\sigma(j)=i$ and therefore $\sigma$ is an transposition if and only if $h=(i,j)$. Hence, if $j=i+1$ and  $i-1$, then  $T\cap (1,i)H(1,j)$ is equal to  $ \{(i,i+1)\} $ and $ \{(i-1,i)\} $, respectively, and otherwise $T\cap (1,i)H(1,j)=\varnothing$. This completes the proof.
\end{proof}
\begin{thm}\label{systemineq}
	Let $p\geq 7$ be a prime number and consider the $p\times p$ matrix $$M=\begin{pmatrix}
	p-1   &1    &0     &0     &\dots &0  \\
	1     &p-2  &1     &0     &\dots &0    \\
	0     &1    &p-2   &1     &0     &0    \\
	\vdots&\dots&\ddots&\ddots&\ddots&\vdots\\
	0     &0    &\dots &1     &p-2   &1     \\
	0     &0    &\dots &0     &1     &p-1
	\end{pmatrix}.$$ Consider the system of inequalities $M(x_1,\ldots ,x_{p})^t\leq (p-1)! \mathbf{1}$ with $(x_1,\ldots ,x_{p})^t\geq \mathbf{0}$ and $x_i$ are integers. Let $x_{\max}:=\max\{x_i\;|\; i=1,\dots,p\}$. Then
\begin{enumerate}
\item $|\{i\in[p] \;|\; x_i\leq \frac{(p-1)!}{p}\}|\geq \lceil\frac{p}{3}\rceil$.
\item If $\sum_{i=1}^p x_i=(p-1)!-k$, then $|\{i \;|\; x_i=x_{\max}\}|\geq p-k-2$.
\item  $\sum_{i=1}^p x_i \leq (p-1)!-\lceil \frac{p}{3}\rceil+2$	
\end{enumerate} 
\end{thm}
\begin{proof}
Let $\mathcal{A}:=\{i\in[p] \;|\; x_i\leq \frac{(p-1)!}{p}\}$ and $\mathcal{B}:=\{i \;|\; x_i=x_{\max}\}$. Consider the partition $\{\{1,2\},\{3,4,5\},\{6,7,8\},\dots,\{p-2,p-1,p\}\}$ of $[p]$ if $p\equiv 2 \mod 3$ and the partition $\{\{1,2\},\{3,4,5\},\{6,7,8\},\dots,\{p-4,p-3,p-2\},\{p-1,p\}\}$ if $p\equiv 1 \mod 3$.  Each  member of partitions corresponds to an obvious inequality, e.g.
$\{1,2\}$ and $\{p-2,p-1,p\}$ are respectively corresponding to  $(p-1)x_1+x_2\leq (p-1)!$ and  $x_{p-2}+(p-2)x_{p-1}+x_p\leq (p-1)!$. Each  inequality corresponding to a member $P$ of  the partitions forces $x_i\leq (p-1)!/p$ for some $i\in P$, where $x_i=\min\{x_j\;|\; j\in P\}$. Since the size of both partitions is  $\lceil \frac{p}{3}\rceil$, we have that $|\mathcal{A}|\geq \lceil\frac{p}{3}\rceil$ and so the first part is proved.

It follows from $M(x_1,\ldots ,x_{p})^t\leq (p-1)! \mathbf{1}$  and  $(x_1,\ldots ,x_{p})^t\geq \mathbf{0}$ that $0\leq \sum_{i=1}^p M_i\mathbf{x}=p(\sum_{i=1}^px_i)\leq p!$, where $M_i$ is $i$-th row of $M$ and so $0\leq \sum_{i=1}^px_i\leq (p-1)!$.  Let $\ell\in [p]$ be such that $x_\ell=x_{\max}$. Thus $\sum_{i=1, i\neq \ell-1,\ell+1}^p (x_{\ell} -x_i)=x_{\ell-1}+(p-2)x_{\ell}+x_{\ell+1}-\sum_{i=1}^{p}x_i\leq (p-1)!-((p-1)!-k)$.
 Thus $\sum_{i=1, i\neq \ell-1,\ell+1}^p (x_\ell -x_i) \in \{0,1,\dots,k\}$. It follows that $|\{i \;|\; x_i< x_{\max}\}|\leq k+2$ and so  $|\mathcal{B}|\geq p-k-2$ and the second part is proved.

Let $\sum_{i=1}^p x_i=(p-1)!-k$ and suppose, for a contradiction, that $k<\lceil \frac{p}{3} \rceil -2$. So $|\mathcal{B}|\geq p-\lceil \frac{p}{3} \rceil+1$ and therefore
$$|\mathcal{A}\cap \mathcal{B}|\geq |\mathcal{A}|+|\mathcal{B}|-p\geq \lceil \frac{p}{3} \rceil+p-\lceil \frac{p}{3} \rceil+1-p\geq 1. $$ 

 Hence $\mathcal{A}\cap \mathcal{B}\neq \varnothing$ and  $x_{\max}\leq (p-1)!/p$. Since $p$ is  prime, by Wilson theorem \cite[P. 27]{wilson} $(p-1)!\equiv -1 \mod p$. Since $x_{\max}$ is integer, we have that $x_i\leq \frac{(p-1)!+1}{p}-1$ for all $i\in [p]$. Therefore $$\sum_{i=1}^px_i=(p-1)!-k\leq p(\frac{(p-1)!+1}{p}-1)=(p-1)!+1-p$$ and so $$p\leq k+1<\lceil \frac{p}{3} \rceil -1,$$ which is a contradiction. So we must have $k\geq \lceil \frac{p}{3} \rceil -2$. This completes the proof.
\end{proof}
In the following  we will prove Theorem \ref{mainresult}.\\
\noindent \textbf{ Theorem:} For all primes $p\geq 11$,   $P(p,3)\leq (p-1)!-\lceil \frac{p}{3}\rceil+2\leq (p-1)!-2$.
\begin{proof}
Let $C$ be a code in $S_p$ with minimum Kendall $\tau$-distance 3. Let $H$ be the Young subgroup of $S_p$ corresponding to the partition $\lambda:=(p-1,1)$ and $X$ be the set of  right cosets of $H$ in $S_p$. If $S=\{(i,i+1)\,|\,1\leq i\leq p-1\}$ and $T:=S\cup \{\xi\}$,  then by Lemma \ref{matrix}, $ \widehat{T^{\rho_{X}^{S_n}}} $ is a  conjugate by a permutation matrix of the matrix $M$ in Theorem \ref{systemineq}. Now Theorem \ref{intprogram} implies that the optimal value of the objective function of the following integer programming problem  gives an upper bound on $|C|$
\begin{align*}\label{max}
\text{Maximize} &\quad \sum_{i=1}^{p}{x_i},\\
\text{subject to}&\quad M(x_1,\ldots ,x_{p})^t \leq |H| \mathbf{1}=(n-1)!\mathbf{1},\\
& \quad x_i\in \mathbb{Z},\,\, x_i\geq 0, \,\, i\in\{1,\ldots,p\},
\end{align*}
where $\textbf{1}$ is a column vector of order $p\times 1$ whose  entries are equal to $1$. Therefore, the result follows from  Theorem  \ref{systemineq}. This completes the proof.
\end{proof} 
\begin{thm}
If $n$ is equal to $6$, $7$, $11$, $13$ and $17$, then $P(n,3)$ is less than or equal to $116$, $716$, $10!-10$, $12!-12$  and $16!-5$, respectively.
\end{thm}
\begin{proof}
Let $S:=\{(i,i+1)\,|\,1\leq i\leq n-1\}$. In view of Theorem \ref{intprogram}, we have used  CPLEX software \cite{cplex} and GAP software \cite{gap} to determine the upper bound for $P(n,3)$ obtained from solving the integer programming problem corresponding to the subgroup $H$ of $S_n$, where $H$ is the Young subgroup corresponding to the partition $(2,2,2)$, when $n=6$, $(5,1,1)$, when $n=7$, $(9,2)$, when $n=11$, $(11,2)$, when $n=13$  and $(16,1)$, when $n=17$. For each of the above subgroups, Using GAP software \cite{gap},  first, we  determined the matrix $\widehat{(T)^{\rho_{X}^{S_{n}}}}$, where  $X$ is the set of  right cosets of $H$ in $S_n$ and  $T:=S\cup \{\xi\}$, then using  CPLEX software \cite{cplex}, we solved the integer programming problem corresponding to the subgroup $H$. 
\end{proof}
 To prove the non-existence of 1-perfct codes in $S_{14}$ and $S_{15}$, we are using  techniques in \cite{white} which is stated in the following proposition.
\begin{prop}\label{white}
\cite[Theorem 2.2]{white} Let $S=\{(i,i+1)\,|\,1\leq i\leq n-1\}$ and $T:=S\cup \{\xi\}$. If $S_n$ contains a subgroup $H$ such that $n\nmid |H|$ and $\widehat{(T)^{\rho_{X}^{S_{n}}}}$ is invertible, where $X$ is the set of  right cosets of $H$ in $S_n$, then  $S_n$ contains no  $1$-perfect codes.
\end{prop}
\begin{thm}
There are no $1$-perfect codes under the Kendall $\tau$-metric in $S_n$ when $n\in\{14,15\}$.
\end{thm}
\begin{proof}
Let $S=\{(i,i+1)\,|\,1\leq i\leq n-1\}$ and $T:=S\cup \{\xi\}$. By Proposition \ref{white}, to prove the non-existence of $1$-perfect codes under the Kendall $\tau$-metric in $S_n$,
we need to consider the young subgroups $H$ of $S_n$, $n\in\{14,15\}$, with two properties: (1) $n\nmid |H|$; (2) the matrix $\widehat{(T)^{\rho_{X}^{S_{n}}}}$  is invertible. Since $\widehat{(T)^{\rho_{X}^{S_{n}}}}$ is a matrix of dimension $n!/|H|$,  by choosing $H$ with a larger size, the dimension of the matrix $\widehat{(T)^{\rho_{X}^{S_{n}}}}$ decreases.
 In the case $n=14$, we consider the  Young subgroup $H$  corresponding to the partition $(6,6,2)$. It is clear that $14\nmid |H|=6!6!2!$. Also, by a software check the matrix $\widehat{(T)^{\rho_{X}^{S_{14}}}}$ which is a matrix of dimension 84084 is invertible and so there are no $1$-perfect codes under the Kendall $\tau$-metric in $S_{14}$. In the case $n=15$, the largest Young subgroup $H$ of $S_{15}$ which satisfies the condition (1) is the Young subgroup corresponding to the partition $\lambda:=(4,4,4,3)$. In this case the matrix $\widehat{(T)^{\rho_{X}^{S_{15}}}}$ is of dimension $1051050$ that the software was unable to check its invertibility and so we used Theorem \cite[Corollary 2.2.22]{JK} to check its invertibility.   By \cite[Corollary 2.2.22]{JK}, if for all partitions $\mu$ of $n$ which $\mu\unlhd \lambda$, $\widehat{T^{\rho_{\mu}}}$ are invertible, where $\rho_{\mu}$ is the irreducible representation of $S_{15}$ corresponding to $\mu$, then $\widehat{T^{\rho_{X}^{S_{15}}}}$ is invertible.  There exist 54 partitions of 15 which dominates the partition $\lambda$. By software check, for each partition $\mu$ of these 54 partition the matrix $\widehat{T^{\rho_{\mu}}}$  is invertible (Table \ref{t1} shows the dimension and the eigenvalue with smallest absolute value of theses martices) and so $\widehat{(T)^{\rho_{X}^{S_{15}}}}$ is invertible  and this completes the proof.
\end{proof}
\begin{con}\label{con}
If $H$ is the Young subgroup corresponding to the partition $(p-1,p-1,2)$ of $S_{2p}$, where $p\geq 3$ is a prime number,  and $X$ is the set of  right cosets of $H$ in $S_{2p}$, then $\widehat{(S\cup \{\xi\})^{\rho_{X}^{S_{2p}}}}$ is invertible. In particular,  there is no $1$-perfect permutation code of length $2p$ with respect to the Kendall $\tau$-metric.
\end{con}
We note  that by software checking Conjecture \ref{con} holds valid for $p\in\{3,5,7\}$.
\begin{longtable}{cccc}
\caption{The dimention and the eigenvalue with smallest absolute  value of the matrix $\widehat{T^{\rho_{\mu}}}$  for all partitions $\mu$ of 15 which dominates the partition (4,4,4,3).}\label{t1}\\
item & Partition& Dimension & Eigenvalue \\\hline
1& (15)& 1 & 1.5$ \times 10 $\\
2 &(14, 1)& 14& 1.104$ \times 10 $\\
3& (13, 2)& 90& 7.232\\
4& (13, 1, 1)& 91 & 7.217\\
5& (12, 3)& 350& 3.686\\
6& (12, 2, 1)& 715& 3.619\\
7& (11, 4)& 910& 5.467$ \times 10^{-1} $\\
8 &(8, 7)& 1430 &-8.095$ \times 10^{-3} $\\
9& (10, 5)& 1638& -2.611$ \times 10^{-3} $\\
10 &(11, 2, 2)& 1925& 3.149$ \times 10^{-1} $\\
11 &(9, 6)& 2002 &-4.503$ \times 10^{-3} $\\
12 &(11, 3, 1)& 2835& 3.745$ \times 10^{-1} $\\
13 &(7, 7, 1)& 5005& 1.035$ \times 10^{-2} $\\
14 &(5, 5, 5)& 6006 &-1.497$ \times 10^{-3} $\\
15 &(10, 4, 1)& 7007& -7.028$ \times 10^{-3} $\\
16 &(10, 3, 2)& 9100& 1.12$ \times 10^{-2} $\\
17 &(9, 5, 1)& 11375 & 2.99$ \times 10^{-3} $\\
18 &(8, 6, 1)& 11583& -4.224$ \times 10^{-3} $\\
19 &(9, 3, 3)& 12740 &-4.444$ \times 10^{-3} $\\
20 &(9, 2, 2, 2)& 13650& -1.825$ \times 10^{-4} $\\
21 &(6, 6, 3)& 21450& -$1.728 \times 10^{-6}$\\
22 &(9, 4, 2)& 22113& 9.626$ \times 10^{-5} $\\
23 &(4, 4, 4, 3)& 24024& -8.294$ \times 10^{-4} $\\
24 &(7, 6, 2)& 25025& -2.424$ \times 10^{-4} $\\
25 &(7, 4, 4)& 25025& 1.108$ \times 10^{-3} $\\
26 &(6, 5, 4)& 30030& 1.033$ \times 10^{-5} $\\
27 &(8, 5, 2)& 32032& 8.217$ \times 10^{-4} $\\
28 &(8, 4, 3)& 35035& -2.925$ \times 10^{-4} $\\
29 &(7, 5, 3)& 45045& 3.142$ \times 10^{-5} $\\
30 &(6, 3, 3, 3)& 50050& 3.128$ \times 10^{-5} $\\
31 &(8, 3, 2, 2)& 58968& 3.477$ \times 10^{-4} $\\
32 &(5, 4, 3, 3)& 75075& -1.733$ \times 10^{-4} $\\
33 &(5, 4, 4, 2)& 81081& -5$ \times 10^{-5} $\\
34 &(7, 3, 3, 2)& 90090& 2.1$ \times 10^{-5} $\\
35 &(5, 5, 3, 2)& 96525& -5.987$ \times 10^{-5} $\\
36 &(6, 5, 2, 2)& 100100& 2.946$ \times 10^{-5} $\\
37 &(7, 4, 2, 2)& 112112& -6.787$ \times 10^{-5} $\\
38 &(6, 4, 3, 2)& 175175& -3.594$ \times 10^{-5} $\\
39 &(12, 1, 1, 1)& 364 & 3.599\\
40 &(11, 2, 1, 1)& 2925& 2.852$ \times 10^{-1} $\\
41 &(10, 2, 2, 1)& 9450& -1.485$ \times 10^{-4} $\\
42 &(10, 3, 1, 1)& 11088& 5.085$ \times 10^{-3} $\\
43 &(9, 4, 1, 1)& 25025& -1.767$ \times 10^{-4} $\\
44 &(7, 6, 1, 1)& 27027& 3.757$ \times 10^{-4} $\\
45 &(8, 5, 1, 1)& 35100& -7.440$ \times 10^{-5} $\\
46 &(9, 3, 2, 1)& 42042& 6.633$ \times 10^{-4} $\\
47 &(6, 6, 2, 1)& 50050& -1.680$ \times 10^{-5} $\\
48 &(5, 5, 4, 1)& 54054& 1.934$ \times 10^{-4} $\\
49 &(8, 3, 3, 1)& 57330& 9.513$ \times 10^{-5} $\\
50 &(6, 4, 4, 1)& 80080 &-2.972$ \times 10^{-5} $\\
51 &(8, 4, 2, 1)& 91000& -2.590$ \times 10^{-5} $\\
52 &(7, 5, 2, 1)& 108108& -3.672$ \times 10^{-5} $\\
53 &(6, 5, 3, 1)& 128700& $-1.920 \times 10^{-5} $\\
54 &(7, 4, 3, 1)& 135135 & $-2.627 \times 10^{-6}$
\end{longtable}
\section{Conclusion}
Due to the applications of PCs under the Kendall $\tau$-metric in flash memories, they have attracted the attention of many researchers. In this paper, we consider the upper bound of  the size of the largest PC with minimum Kendall
$\tau$-distance 3. Using group theory, we formulate an integer programming problem that is depend on a non-trivial subgroup of $S_n$ of choice, where the optimal value of the objective function  gives an upper bound on $P(n,3)$.  After that, by solving the integer programming problem corresponding to some subgroups of $S_n$, when $n\geq 7$ is a prime number or $n\in\{6,14,15\}$, we  improve the upper bound on $P(n,3)$.

\section{Declarations}

\subsection{Ethical Approval and Consent to participate}

Not applicable.

The current manuscript does not report on or involve the use of any animal or human data or tissue. 

\subsection{Consent for publication}

Not applicable.

The current manuscript does not contain data from any individual person. 

\subsection{Availability of supporting data}

All data generated or analysed during this study are included in this published article. 

\subsection{Competing interests}

The authors declare that they have no competing interests.

\subsection{Funding}

This research is supported by the Deputy of Research and Technology of University of Isfahan under a grant given to the research group CSG (Code-Scheme-Group). 
F. Parvaresh is also supported by IPM in part by grant No. 1401680050.

These funding sources had no role in the design of this study and will not have any role during its execution, analyses, interpretation of the data, or decision to submit results.

\subsection{Authors' contributions}

 All authors read and approved the final manuscript.

\end{document}